\documentclass[12pt, oneside]{amsart}

\usepackage{geometry}   
\usepackage[ruled,vlined]{algorithm2e}
\usepackage{color}
\usepackage{bm}
\usepackage{graphicx}
\usepackage[justification=centering]{caption}
\usepackage{amsthm,amssymb,amsmath}
\usepackage{subfigure}
\allowdisplaybreaks[3]

\newtheorem{theorem}{Theorem}[section]
\newtheorem{lemma}[theorem]{Lemma}

\newcommand{\anote}[1]{\textcolor{blue}{(Akbar: #1)}}
\newcommand{\ynote}[1]{\textcolor{red}{(Yifan: #1)}}

\def\e{\mathsf{e}}

\def\T{\mathfrak{T}}
\def\SQ{\mathcal{Q}}
\def\G{\mathcal{G}}

\def\Er{\mathfrak{Er}}
\def\nm{\mathsf{NM}}

\def\MM{\mathsf{M}}
\def\M{\mathcal{M}}

\def\C{\mathfrak{C}}

\newcommand{\CCase}[2]{\noindent {\bf Case #1:} \emph{#2}}

\title[Yifan Jing and Akbar Rafiey]{Counting maximal near perfect matchings in  quasirandom and dense graphs}

\author{Yifan Jing}
\address{%
Department of Mathematics\\
University of Illinois at Urbana-Champaign\\
Urbana, IL, USA}
\email{yifanjing17@gmail.com}

\author{Akbar Rafiey}
\address{%
Department of Computing Science\\
Simon Fraser University\\
Burnaby, BC, Canada}
\email{arafiey@sfu.ca}
\begin{document}
\maketitle
\begin{abstract}

A \emph{maximal $\varepsilon$-near perfect matching} is a maximal matching which covers at least $(1-\varepsilon)|V(G)|$ vertices. In this paper, we study the number of maximal near perfect matchings in generalized quasirandom and dense graphs. We provide tight lower and upper bounds on the number of $\varepsilon$-near perfect matchings in generalized quasirandom graphs. Moreover, based on these results, we provide a deterministic polynomial time algorithm that for a given dense graph $G$ of order $n$ and a real number $\varepsilon>0$, returns either a conclusion that $G$ has no $\varepsilon$-near perfect matching, or a positive non-trivial number $\ell$ such that the number of maximal $\varepsilon$-near perfect matchings in $G$ is at least $n^{\ell n}$. Our algorithm uses algorithmic version of Szemer\'edi Regularity Lemma, and has $O(f(\varepsilon)n^{5/2})$ time complexity. Here $f(\cdot)$ is an explicit function depending only on~$\varepsilon$.

\end{abstract}



\section{Introduction}
For a simple graph $G=(V,E)$, a \emph{matching} $\M$ of $G$ is a subset of $E(G)$ such that the edges in $\M$ do not have common end vertices. We say $\M$ is a \emph{perfect matching} if $|\M|=|V(G)|/2$. The problem of computing the total number of perfect matchings in a graph, has been extensively studied by mathematicians and computer scientists. It is known that the number of perfect matchings in a bipartite graph is equivalent to the permanent of its adjacency matrix. See \cite{bapat2007recent} for a recent survey on several theorems and open problems on permanent
of matrices and its algebraic properties. The evaluation of the permanent has attracted the attention of researchers for almost two centuries, however,  despite many attempts, an efficient algorithm for general matrices has proved elusive. Indeed, Ryser's algorithm \cite{ryser1963combinatorial} remains the most efficient for computing the permanent exactly, even though it
uses as many as $\Theta(n2^n)$ arithmetic operations. A notable breakthrough was achieved
about 60 years ago with the publication of Kasteleyn's algorithm for counting
perfect matchings in planar graphs \cite{kasteleyn1961statistics}, which uses just $O(n^3)$ arithmetic
operations. 

It turns out that computing the number of perfect matchings in a bipartite graph (computing permanent of a $\{0,1\}$-matrix) falls into the $\#$P-complete complexity class \cite{valiant1979complexity}, and thus, modulo a basic complexity theoretic conjecture, cannot be solved (exactly) in polynomial time. This remains true even for 3-regular bipartite graphs~\cite{dagum1992approximating}, and for bipartite graphs with minimum vertex degree at least $n/2$~\cite{broder1986hard}. Using the so-called Pfaffian orientations, the perfect matchings in a planar graph can be counted in polynomial time \cite{fisher1961statistical,kasteleyn1961statistics,temperley1961dimer}. A generalization of this approach yields a polynomial time algorithm for graphs of bounded genus \cite{galluccio1999theory,tesler2000matchings}. Furthermore, we can count the perfect matchings in a graph of bounded treewidth \cite{arnborg1991easy}. Basically, most of the
positive results are concerned with sparse graphs. For other
graph classes, less is known, but $\#$P-completeness is known for chordal and chordal bipartite graphs~\cite{okamoto2009counting}.

Ever since the introduction of the $\#P$ complexity class by Valiant \cite{valiant1979complexity}, the focus on these problems shifted to finding approximate solutions. Jerrum,
Sinclair, and Vigoda \cite{jerrum2004polynomial} in a breakthrough obtained a fully polynomial time randomized approximation scheme (FPRAS) for the permanent of matrices with nonnegative entries. In other words, they designed a randomized algorithm that for any given  $\varepsilon > 0$, outputs a $1 + \varepsilon$ multiplicative approximation
of the permanent, in time polynomial in $n$ and $1/\varepsilon$. This approach focuses on rapidly mixing Markov chains to obtain appropriate random samples. Many
randomized approximation schemes for various counting problems were derived in this way -- see e.g.,
\cite{jerrum2003counting,jerrum1996markov,sinclair2012algorithms} for several nontrivial applications. Unfortunately, Jerrum, Sinclair and Vigoda's \cite{jerrum2004polynomial} approach seems too complicated to be used in practice and
the approach does not appear to extend to nonbipartite graphs, since odd cycles are problematic. For this reason, a simpler Markov chain was proposed in \cite{diaconis2001statistical,dyer2017switch}. In \cite{dyer2017switch}, counting all perfect matchings in some particular classes of bipartite graphs was examined. Recently, Dyer and M$\ddot{\text{u}}$ller~\cite{dyer2017counting} extended the analyses in \cite{dyer2017switch} to hereditary classes of nonbipartite graphs.

There are only a few results concerning approximately counting perfect matchings in general graphs. Jerrum and Sinclair~\cite{jerrum1989approximating} considered this problem in general graphs. Their Markov chain method requires exponential time complexity for general graphs. More precisely, their method requires time polynomial in the ratio of number of near perfect matchings and number of perfect matchings, which may be exponential in the size of graph. This condition is satisfied for graphs with $2n$ vertices and minimum degree at least $n$, therefore providing a FPRAS for this class of graphs. There have been other approaches to tackle the problem. Chien \cite{chien2004determinant} presents a determinant-based algorithm for the number of perfect matchings in a general graph. His estimator requires $O(\varepsilon^{-2}3^{n/2})$ trials to obtain a $(1\pm \varepsilon)$-approximation of the correct
value with high probability on a graph with $2n$ vertices, and a polynomial number $(O(\varepsilon^{-2}n\omega(n))$ of trials on random graphs, where $\omega(n)$ is any function tending to infinity. Refer to~\cite{furer2005approximately} for a simpler algorithm with experimental results. 

There are results concerning counting total number of matchings (not only perfect matchings) in graphs and random graphs. Vadhan in \cite{vadhan2001complexity} showed that the problems of counting matchings remain hard when restricted to planar bipartite graphs of bounded degree or regular graphs of constant degree. Therefore, approximating this number has been studied by researchers. For example, Bayati \textit{et al.}~\cite{bayati2007simple} construct a deterministic fully polynomial time approximation scheme (FPTAS) for computing
the total number of matchings in a bounded degree graph. Additionally, for an arbitrary graph, they construct a deterministic algorithm for computing approximately the number of matchings within running time $exp(O(\sqrt{n}\log^2n))$, where $n$ is the number of vertices. Patel and Regts \cite{PR17} recently provided an alternative deterministic algorithm to approximately count matchings in bounded degree graphs. This is the same result as in \cite{bayati2007simple}, using a completely different method.
Zdeborov\'{a} and M\'{e}zard \cite{zdeborova2006number} considered this problem on sparse random graphs, in fact, their result is the computation of the entropy, i.e. the leading order of the logarithm of the number of solutions, of matchings with a given size.

In terms of lower bounds, Schrijver \cite{schrijver1998counting} shows that any $d$-regular bipartite graph with $2n$ vertices has at least\[\Big(\frac{(d-1)^{d-1}}{d^{d-2}}\Big)^n\]
perfect matchings. More generally, let $m_k(G)$ denote number of matchings of size $k$ in graph $G$. Friedland, Krop and Markstr$\ddot{\text{o}}$m \cite{friedland2008number} conjectured the following lower bound on $m_k(G)$ where $G$ is a $d$-regular bipartite graph
\[m_k(G) \geq \binom{n}{k}^2\Big(\frac{d-n/k}{d}\Big)^{n(d-n/k)}(dn/k)^{n^2/k}
\]
The conjecture was proved in \cite{csikvari2014lower} and extended to irregular bipartite graphs in \cite{lelarge2017counting}. 

Given the difficulty of counting number of perfect matchings, in particular beyond bipartite graphs, we turn our attention to near perfect matchings. In this paper, we focus on counting the number of maximal near perfect matchings in graphs. A matching $\M$ in $G$ is \emph{maximal} if the graph induced by the vertices which are not in $\M$ is empty. Counting maximal matchings is $\#P$-complete even in bipartite graphs with maximum degree five~\cite{vadhan2001complexity}. To the best of our knowledge there is no result concerning approximating the number maximal matchings. A \emph{maximal $\varepsilon$-near perfect matching} is a maximal matching that covers at least $(1-\varepsilon)|V(G)|$ vertices. Let $\nm(G,\varepsilon)$ denote the number of maximal $\varepsilon$-near perfect matchings in graph $G$. Our first result is an approximation on the number of near perfect matchings in $\varepsilon$-regular graphs.

\begin{theorem}\label{thm:main}
Given $\varepsilon>0$ and a bipartite $\varepsilon$-regular graph $G$ with density $p$. Then there exists $n_0=n_0(\varepsilon,p)$, such that for $|V(G)|=2n>n_0$, we have
\[
(1-3\sqrt\varepsilon)n\log pn\leq \log\nm(G,\sqrt\varepsilon)\leq (1+3\sqrt\varepsilon)n\log pn.
\]
\end{theorem}

Let $P$ be a symmetric $m\times m$-matrix, such that $0\leq p_{i,j}\leq 1$, where $p_{i,j}$ is the $(i,j)$-entry of $P$. We define generalized quasirandom graphs as follows. A graph $G\in\SQ(n^{(m)},P,\varepsilon)$ if $V(G)=\bigsqcup_{i=1}^m V_i$ with $|V_1|=\dots=|V_m|=n$, and for every $i\neq j$, $(V_i,V_j)$ is $\varepsilon$-regular with density $p_{ij}$, and $G[V_i]$ is $\varepsilon$-close (in the sense of cut metric) to a random graph $\G(n,p_i)$ for every $i$. Here $p_{i,j}$ is the $(i,j)$-entry of $P$, and $p_i$ is $(i,i)$-entry of $P$. 

Given $G\in \SQ(n^{(m)},P,\varepsilon)$, define the \emph{quotient graph} $H=G/m$ as a weighted graph, such that $V(H)=[m]$, and the edge weight $u(ij)=p_{ij}$, the vertex weight $u(i)=p_i$. We let $ij\in E(H)$ if $p_{ij}\neq 0$. Let $w:[m]^2\to [0,1]$ be a function, we consider the following linear equations on $H$.
\begin{equation}\label{eq:5}
\sum_{1\leq j\leq m,\, ij\in E(H)}w(ij)=1\qquad\text{ for every }1\leq i\leq m,
\end{equation}

We have the following result on the number of maximal near perfect matchings in generalized quasirandom graphs.

\begin{theorem}\label{thm:multimatching}
Suppose we have an integer $m\geq2$, and a $m\times m$-matrix $P$. Then there exists $n_0>0$ and $c>0$, such that if $n>n_0$ and $\varepsilon<c$, for every  graph $G\in\SQ(n^{(m)},P,\varepsilon)$, let $H$ be the quotient graph of $G$, we have
\begin{enumerate}
\item If the linear system (\ref{eq:5}) of $H$ does not have any solution, $G$ does not have maximal $\sqrt\varepsilon$-near perfect matchings.
\item If the linear system (\ref{eq:5}) of $H$ has solutions, then
\[
(1-4\sqrt\varepsilon)\frac{m}{2}n\log n\leq \log\nm(G,\sqrt\varepsilon)\leq(1+7\sqrt\varepsilon)\frac{m}{2}n\log n.
\]
\end{enumerate}
\end{theorem}

Based on the algorithmic version of Szemer\'edi regularity lemma and the results we obtain for quasirandom graphs, we provide a deterministic polynomial-time algorithm on approximating the number of maximal near perfect matchings in dense graphs. A graph $G$ is called \emph{dense} if $|E(G)|\geq\alpha|V(G)|^2$ for some fixed $\alpha$. Given a dense graph $G$, our  algorithm provide a non-trivial lower bound on the number of maximal near perfect matchings in $G$.
\medskip

\begin{quote}
	{\sc Number of Max Near Perfect Matchings Dense}.\\
	\emph{Input:} A graph $G$ of order $n$ and a real number $\varepsilon>0$.\\
	\emph{Output:} Either a conclusion that $G$ does not contain a maximal $\varepsilon$-near perfect matching, or a (non-trivial) real number $\ell$ such that $\nm(G,\varepsilon)>n^{\ell n}$.
\end{quote}
\medskip

In particular, the lower bound is obtained by the following theorem.

\begin{theorem}\label{thm:1.7}
Let $G$ be a dense graph on $n$ vertices. Then
\[
\log\nm(G,\sqrt\varepsilon)\geq(1-4\sqrt\varepsilon)\sup_{w(\mathbf{e})\in\mathfrak{S}}\sum_{e\in E_4}\frac{w(e)n}{K}\log\frac{w(e)n}{K}+\sum_{e\in E_3}\frac{w(e)n}{K}\log p_e\frac{w(e)n}{K}.
\]
\end{theorem}
The value of $w$ and the set $\mathfrak{S}$ are determined by a linear programming, and the values of $K,p$, the sets $E_3,E_4$ are determined by the algorithm, we will discuss it in details in Section~4.

The paper is organized as follows. In the next section, we give basic definitions and properties in graph theory, and the theoretical background used in the paper. In Section 3, we discuss the matchings in generalized quasirandom graphs. In Section~4, we consider the problem for the dense graphs, and provide an approximation algorithm.

\section{Preliminaries}

We will use standard definitions and notation in graph theory. Given a graph $G$, a \emph{matching} $\M$ of $G$ is a subset of $E(G)$ such that the edges in $\M$ do not have common end vertices. We say $M$ is a \emph{perfect matching} if $|\M|=|V(G)|/2$, and $\M$ is \emph{maximal} if there does not exist another matching $\M_1\neq\M$ such that $\M\subseteq\M_1$. Given a graph $G$ and a real number $\varepsilon>0$, we say a matching $\M$ is $\varepsilon$-\emph{near perfect} if $|M|\geq(1-\varepsilon)|V(G)|/2$. Given a simple graph $G$, let $\nm(G,\varepsilon)$ be the number of maximal $\varepsilon$-near perfect matchings in $G$. We use $[n]$ to denote the set of integers $\{1,\dots,n\}$. All the logarithms in the paper are taken base $e$.

Szemer\'edi Regularity Lemma \cite{szemeredi1975regular} is one of the most powerful tools in modern graph theory. Szemer\'edi first used this lemma in his celebrated theorem on the existence of long arithmetic progressions in dense subset of integers~\cite{S}. The lemma gives us the rough structure of dense graphs. Roughly speaking, Given any dense graph $G$ and the error $\varepsilon>0$, one can partition the vertex set of $G$ into constant (only depending on $\varepsilon$) parts, and the subgraph between each two parts except an $\varepsilon$ fraction performs like a random graph.  To make this precise, we need some definitions.

Given a simple graph $G$ and $X,Y\subseteq V(G)$. Let $\e(X,Y)$ be the number of edges between $X,Y$ then the \emph{edge density} between $X$ and $Y$ is defined as $d(X,Y)=\e(X,Y)/(|X||Y|)$. A pair of vertex subsets $(X,Y)$ is {\em $\varepsilon$-regular} if for all subsets $X^\prime\subseteq X$ and $Y^\prime\subseteq Y$ that satisfy $|X^\prime|\geq\varepsilon|X|$ and $|Y^\prime|\geq\varepsilon|Y|$, we have $|d(X^\prime,Y^\prime)-d(X,Y)|<\varepsilon$.  A pair of vertex set $(X,Y)$ is \emph{$\varepsilon$-regular with density $p$}, if for every $X'\subseteq X$ and $Y'\subseteq Y$ with $|X'|\geq \varepsilon |X|$ and $|Y'|\geq \varepsilon|Y|$, we have $|d(X',Y')-p|\leq \varepsilon$. Note that under this definition, the edge density between $X$ and $Y$ is not necessarily $p$.

We say a vertex partition $\mathcal{P}=\{V_1,\dots,V_K\}$ is \emph{equitable} if $\big||V_i|-|V_j|\big|\leq1$ for every $1\leq i<j\leq K$. An equitable vertex partition $\mathcal{P}$ with $K$ parts is \emph{$\varepsilon$-regular} if all but at most $\varepsilon K^2$ pairs of parts $(V_i,V_j)$ are $\varepsilon$-regular.

\begin{theorem}[Szemer\'edi Regularity Lemma \cite{szemeredi1975regular}]
For every $\varepsilon>0$ and every integer $m$, there exists an integer $M=M(m,\varepsilon)$ such that every simple graph $G$ has an $\varepsilon$-regular partition into $K$ parts, where $m\leq K\leq M$.
\end{theorem}

To obtain a Szemer\'edi partition, there are many known polynomial time algorithms, for example, \cite{N94}. Recently, Tao \cite{Tao10} provided a probabilistic algorithm which produces an $\varepsilon$-regular partition with high probability in constant time (depending on $\varepsilon$). In this paper, we will use a more recent deterministic PTAS due to Fox et al.~\cite{FLZ}.
\begin{theorem}[\cite{FLZ}]\label{thm:algpartition}
There exists an $O_{\varepsilon,\alpha,k}(n^2)$ time algorithm, which, given $\varepsilon>0$, and $0<\alpha<1$, an integer $k$, and a graph $G$ on $n$ vertices that admits an $\varepsilon$-Szemer\'edi partition with $k$ parts, outputs a $(1+\alpha)\varepsilon$-Szemer\'edi partition of $G$ into $k$ parts.
\end{theorem}

Our algorithm to compute a lower bound on the number of maximal near perfect matchings in dense graphs is based on estimating the number of maximal near perfect matchings in quasirandom graphs. Quasirandom graphs are graphs which share many properties with random graphs. The notion of quasirandomness was first introduced in seminal papers by Chung, Graham and Wilson \cite{quasi} and independently by Thomason \cite{quasi2}. In this paper, we will use a slightly different notion of quasirandomness.

Given a simple graph $G$ and $\varepsilon>0$, we say $G$ is \emph{$(\varepsilon,p)$-quasirandom}, denoted by $G\in\SQ(n,p,\varepsilon)$, if for every $X,Y\subseteq V(G)$ that satisfy $X\cap Y=\varnothing$ and $|X|\geq\varepsilon|V(G)|$, $|Y|\geq\varepsilon|V(G)|$, we have $|d_G(X,Y)-d_{K_{n,p}}(X,Y)|<\varepsilon$, where $K_{n,p}$ is an edge weighted complete graph on $V(G)$ with edge weight $p$. Here for a weighted graph, we define $\e(X,Y)=\sum_{e\in E(X,Y)}w(e)$, and the edge density $d(X,Y)=\e(X,Y)/(|X||Y|)$. We say a graph $G$ is \emph{generalized quasirandom}, if there is an equitable vertex partition of $V(G)$, such that the graphs induced on each part, and between every two different parts, are "random like". To be more precise, let $P$ be a symmetric $m\times m$-matrix, such that $0\leq p_{ij}\leq 1$, where $p_{ij}$ is the $(i,j)$-entry of $P$. A graph $G\in\SQ(n^{(m)},P,\varepsilon)$ if $V(G)=\bigsqcup_{i=1}^m V_i$ such that $|V_i|=n$ and $(V_i,V_j)$ is $\varepsilon$-regular with density $p_{ij}$, and $G[V_i]$ is $(\varepsilon,p_i)$-quasirandom, where $p_{i}$ is the $(i,i)$-entry of $P$.
\section{Matchings in generalized quasirandom graphs}

\subsection{Matchings in quasirandom graphs}

By the definition of $\varepsilon$-regular, we have the following lemma.
\begin{lemma}\label{lem:reg}
Suppose $|X|=|Y|=n$ and $(X,Y)$ is $\varepsilon$-regular with density $p$. Then
\[
\big|\{v\in X\mid (1-\delta)pn\leq d(v)\leq (1+\delta)pn\}\big|\geq (1-2\varepsilon)n,
\]
where $\delta=\varepsilon/p$.
\end{lemma}
\begin{proof}
Let $X'\subseteq X$ such that for every $v\in X'$, we have $d(v)>(1+\delta)pn$. Thus $\e(X',Y)\> (1+\delta)pn|X'|=(p+\varepsilon)|X'||Y|$. On the other hand, if $|X'|\geq \varepsilon n$, since $(X,Y)$ is $\varepsilon$-regular with density $p$, this gives us $\e(X',Y)\leq (p+\varepsilon)|X'||Y|$, contradiction.
\end{proof}
We are now going to prove Theorem~\ref{thm:main}. We suggest that reader consults Algorithm~\ref{alg-quasi-bi} while reading the proof.

\begin{proof}[Proof of Theorem \ref{thm:main}]
Suppose $V(G)$ has a bipartition $X,Y$, with $|X|=|Y|=n$, and the edge density between $X,Y$ is $p$. Since $(X,Y)$ is $\varepsilon$-regular, it is $\sqrt\varepsilon$-regular. Let $\T(X)$ be the set of \emph{typical} vertices in $X$, that is, set of vertices $v$ such that $(1-\delta)pn\leq d(v)\leq (1+\delta)pn$, where $\delta=\sqrt\varepsilon/p$.

We first consider the lower bound. Count the number of near perfect matchings greedily in $k+t$ phases, the values of $k$ and $t$ will be determined later. In {\sc Phase~1}, Let $X_1:=X$, and we pick an arbitrary vertex $v\in\T(X_1)$. Pick a vertex $u\in N(v)$ arbitrarily. Let $X_1:=X-v$, $G:=G-\{u,v\}$ and $\T(X_1):=\T(X_1)-v$. Keep doing this procedure $\sqrt \varepsilon n$ times, that is, we stop {\sc Phase 1} after removing $\sqrt \varepsilon n$ vertices from $X_1$. We denote the remaining vertices in $X_1$ by $X_2$, and move to {\sc Phase 2}.

In {\sc Phase 2}, let $$\T(X_2):=\{v\in X_2\mid  (1-\delta)pn(1-\sqrt\varepsilon)\leq d(v)\leq (1+\delta)pn(1-\sqrt\varepsilon)\}.$$ We pick a vertex $v\in \T(X_2)$ arbitrarily and pick $u\in N(v)$. Let $X_2:=X_2-v$, $G:=G-\{u,v\}$ and $\T(X_2):=\T(X_2)-v$. We keep doing this procedure $\sqrt \varepsilon n$ times, then we let $X_3:=X_2$ and move to {\sc Phase 3}, and we similarly let $$\T(X_3):=\{v\in X_3\mid  (1-\delta)pn(1-2\sqrt\varepsilon)\leq d(v)\leq (1+\delta)pn(1-2\sqrt\varepsilon)\}.$$

Suppose after applying {\sc Phase k}, we have $|X_{k+1}|\leq c\sqrt\varepsilon n$ in {\sc Phase k+1}, where $c=\frac{1}{(1-\delta)p}$. In {\sc Phase k+1}, we pick a vertex in $\T(X_{k+1})$ and remove it as well as one of its neighbor as we did before. But now, instead of repeating this procedure $\sqrt\varepsilon n$ times, we do it $(1-\delta)p|X_{k+1}|$ times. Then we move to {\sc Phase k+2}. We run the algorithm in {\sc Phase k+2} $(1-\delta)p|X_{k+2}|$ times and then move to the next phase. We stop the algorithm after {\sc Phase k+t}, if in {\sc Phase k+t+1} we have $|X_{k+t+1}|\leq \sqrt\varepsilon n$. See {\sc Algorithm} \ref{alg-quasi-bi} for the algorithm.

The algorithm is well-defined, since in {\sc Phase i} for each $i$, the graph on $(X_i,Y_i)$ is $\sqrt\varepsilon$-regular with density $p$. By Lemma \ref{lem:reg}, we can always define the set $\T(X_i)$. We ignore the floor and ceiling function here to simplify the computation. By the way we define $k$ and $t$, we have
\[
k=\frac{1}{\sqrt\varepsilon}-c, \qquad t=\frac{\log(1-\delta)p}{\log(1-(1-\delta)p)}.
\]

Note that the collection of edges we removed in each steps in the algorithm gives us a $\sqrt\varepsilon$-near perfect matching in $G$. Therefore,
\[
\nm(G,\sqrt\varepsilon)\geq\prod_{i=0}^{k-1}\frac{\big((1-\delta)p(n-i\sqrt\varepsilon n)\big)!}{\big((1-\delta)p(n-i\sqrt\varepsilon n)-\sqrt\varepsilon n\big)!}\prod_{i=0}^{t-1}\big((1-\delta)pc\sqrt\varepsilon n(1-(1-\delta)p)^i\big)!.
\]
The first product counts the number of possible different collections of edges we removed from first $k$ phases, and the second product counts the number of different collections of edges we removed from the last $t$ phases. By a complicated but standard computation (see Lemma~\ref{lem:a1} in Appendix for the computation), we have
\[
\log \nm(G,\sqrt\varepsilon)\geq (1-3\sqrt\varepsilon)n\log pn.
\]

Now we consider the upper bound. Note that all the vertices in $\T(X_1)$ has at most $(1+\delta)pn$ neighbors, we have $\nm(G,\sqrt\varepsilon)\leq \big((1+\delta)pn\big)^{(1-2\sqrt\varepsilon) n}n^{2\sqrt\varepsilon n}$. Therefore,
\[
\log\nm(G,\sqrt\varepsilon)\leq (1+3\sqrt\varepsilon)n\log pn,
\]
finishes the proof.
\end{proof}

\begin{algorithm}[H]
\SetAlgoLined
\SetKwInOut{Input}{Input}
\SetKwInOut{Return}{Return}
\Input{bipartite $\varepsilon$-regular graph $G=(X,Y)$ with density $p$.}
 $c=\frac{1}{(1-\delta)p}, k=\frac{1}{\sqrt{\varepsilon}}-c, t=\frac{\log(1-\delta)p}{\log(1-(1-\delta)p)}$, $\nm=1$,
 $X_0=X$\;
 \For{$i= 1$ to $k$}{
    $X_i = X_{i-1}$\;
    $\T(X_i)=\{v\in X_i\mid  (1-\delta)pn(1-(i-1)\sqrt\varepsilon)\leq d(v)\leq (1+\delta)pn(1-(i-1)\sqrt\varepsilon)\}$\;
    \For{$j= 1$ to $\sqrt{\varepsilon}n$}{
        Pick $v$ from $\T(X_i)$ and pick $u$ from $N(v)$\;
        $X_i = X_i-v$, $G = G-\{u,v\}$, $\T(X_i)=\T(X_i)-v$\;
        $\nm = \nm \times \big[(1-\delta)pn(1-(i-1)\sqrt\varepsilon)-j+1\big]$\;
    }
 }
 
 \For{$i= 1$ to $t$}{
    $X_{k+i}=X_{k+i-1}$\;
    $\T(X_{k+i})=\{v\in X_{k+i}\mid  (1-\delta)pn(1-((k+i)-1)\sqrt\varepsilon)\leq d(v)\leq (1+\delta)pn(1-((k+i)-1)\sqrt\varepsilon)\}$\;
    \For{$j= 1$ to $(1-\delta)p|X_{k+i}|$}{
        Pick $v$ from $\T(X_{k+i})$ and pick $u$ from $N(v)$\;
        $X_{k+i} = X_{k+i}-v$,
        $G = G-\{u,v\}$,
        $\T(X_{k+i})=\T(X_{k+i})-v$\;
        $\nm = \nm \times \big((1-\delta)pc\sqrt\varepsilon n(1-(1-\delta)p)^i-j+1\big)$\;
    }
 }
\Return{$\nm$}
\caption{Near Perfect Matchings in Quasirandom Bipartite Graphs} 
 \label{alg-quasi-bi}
\end{algorithm}

By applying the same greedy procedure, we also obtain a good approximation for the quasirandom graphs. The proof is similar to the proof of Theorem~\ref{thm:main}, and we omit further details.
\begin{theorem}\label{thm:1q}
Suppose $\varepsilon>0$ and $G\in\SQ(n,p,\varepsilon)$. Then there exists $n_0=n_0(\varepsilon,p)$, such that if $|V(G)|=n>n_0$, we have
\[
(1-3\sqrt\varepsilon)\frac{1}{2}n\log pn\leq\log\nm(G,\sqrt\varepsilon)\leq(1+3\sqrt\varepsilon)\frac{1}{2}n\log pn.
\]
\end{theorem}

\subsection{Matchings in generalized quasirandom graphs}

In this subsection, we will focus on generalized quasirandom graphs.  

For a graph $H$ of order $m$, where $V(H)=[m]$, let $w:[m]^2\to[0,1]$ be a symmetric function such that $w(ij)=0$ when $ij\not\in E(H)$. Now consider the following linear equations.
\begin{equation}\label{eq:1}
\sum_{j=1}^mw(ij)=1\qquad\text{ for every }1\leq i\leq m,
\end{equation}
 We write $w(i):=w(ii)$. The following example shows that, the system of linear equations (\ref{eq:1}) may have exactly one solution, or infinitely many solutions, or no solutions, see Figure \ref{fig:equations}. In all the examples, we assume $w(i)=0$ for every vertex $i$.
\begin{figure}[h]
\centering
\includegraphics[width=11cm]{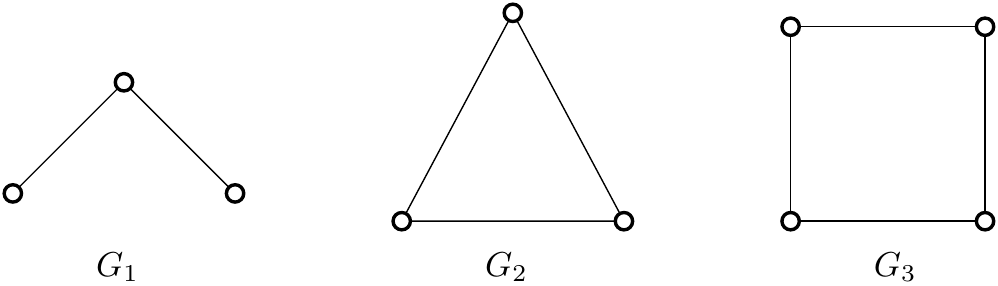}
\caption{Linear equation (\ref{eq:1}) of $G_1$ does not have any solutions, $G_2$ has exactly one solution, and $G_3$ has infinitely many solutions.}\label{fig:equations}
\end{figure}

Given a graph $H$, suppose the linear equation (\ref{eq:1}) does not have any solutions, then for each edge $e$ and vertex $v$ in $H$, we assign values $w(i)\in [0,1]$ and $w(ij)\in[0,1]$ arbitrarily if $ij\in E(H)$. Define the error $\Er(H,w)$ as follows 
\[
\Er(H,w)=\sum_{i=1}^m\Big|1-\sum_{j=1}^mw(ij)\Big|.
\]
Let $\Er(H)=\min_{w:E(H)\to [0,1]}\Er(H,w).$  Suppose $|E(H)|=h$ and let $\bm{x}=(e_1,\dots,e_h)$. We can rewrite~(\ref{eq:1}) as
\[
A\bm{x}=\bm{b}^T,
\]
where $\bm{b}=(1,\dots,1)$ is a $(1\times m)$-vector and $A$ is a $(m\times h)$-matrix. Let 
\begin{equation*}
Q:=\left[
\begin{array}{cc}
A & I_{m\times m}\\
-A & I_{m\times m}
\end{array}\right]
\end{equation*}
where $I$ is the identity matrix. Let $\bm{c}=(\bm{b},-\bm{b})$ be a $(1\times 2m)$-vector, and let $\bm{y}=(\bm{x},y_1,\dots,y_m)$. Then $\Er(H)$ is the solution of the linear programming
\begin{align}
&\min \sum_{i=1}^my_i,\nonumber \\
& Q\bm{y}^T\geq \bm{c}^T,\\
& \bm{0}\leq \bm{y}\leq \bm{1}. \nonumber
\end{align}

Note that $A$ is a $(0,1)$-matrix and the sum of each column of $A$ is $2$. Then for a fixed $m$, there are finitely many possible matrices $A$. This fact implies the following lemma.

\begin{lemma}\label{lem:er}
Given a graph $H$ with $V(H)=[m]$. If the system of linear equations (\ref{eq:1}) does not have any solution on $H$, then there exists a constant $c=c(m)>0$, such that
$
\Er(H)\geq c.
$
\end{lemma}

For graph $G\in\SQ(n^{(m)},P,\varepsilon)$, let $p_{ij}$ be the $(i,j)$-entry of $P$ and $p_i$ be the $(i,i)$-entry of $P$. Suppose $H$ is the quotient graph of $G$, that is, $V(H)=[m]$ and edge $ij$ in $H$ has weight $p_{ij}$ if $p_{ij}\neq0$, vertex $i$ has weight $p_i=p_{ii}$. We are going to prove Theorem \ref{thm:multimatching}.
\begin{proof}[Proof of Theorem \ref{thm:multimatching}]
Suppose $V(G)=V_1\sqcup\dots\sqcup V_m$ and $|V_1|=\dots=|V_m|=n$. We consider linear equation (\ref{eq:1}) on the quotient graph $H$.
\medskip
	
\CCase{1}{Linear system (\ref{eq:1}) does not have any solution.}

By Lemma \ref{lem:er}, given $m$ there is a constant $c(m)>0$ such that $\Er(H)\geq c(m)$. Let $\varepsilon<c^2(m)$. Suppose $G$ has a maximal $\sqrt\varepsilon$-near perfect matching $\M$. Then for every $1\leq i\leq m$, except at most $\sqrt\varepsilon n$ vertices, all the vertices in $V_i$ are covered by edges in $\M$. Now we consider the quotient graph $H$, for every $ij\in E(H)$, define $w(ij)=\frac{|\M\cap E(G[V_i,V_j])|}{n}$ and $w(i)=\frac{|\M\cap E(G[V_i])|}{n}$ for every $i\in [m]$. $\M$ being a $\sqrt\varepsilon$-near perfect matching means $\Er(H,w)<\sqrt\varepsilon<c(m)$, which is a contradiction.
\medskip
	
\CCase{2}{Linear system (\ref{eq:1}) has a unique solution.}

Now for every $e\in E(H)$ and $v\in V(H)$, we have a solution $w$ for the linear system (\ref{eq:1}). For every $1\leq i\leq m$, partition the vertex set $V_i$ into at most $m$ parts $V_{i,0},V_{i,1},\dots,V_{i,i-1},V_{i,i+1},\dots,V_{i,m}$, satisfying  $|V_{i,j}|=w(ij)n$ and $|V_{i,0}|=w(i)n$. Suppose $N(V_i)$ is the number of different ways of partitioning $V_i$, define $w(i0)=w(i)$, then we have
\[
N(V_i)=\frac{n!}{\prod_{0\leq j\leq m,\, j\neq i}^m(w(ij)n)!}.
\]

Note that $G[V_{i,j},V_{j,i}]$ and $G[V_{i,0}]$ is quasirandom for every $ij\in E(H)$ and every $i\in [m]$. Applying Theorems \ref{thm:main} and  \ref{thm:1q} on $G[V_{i,j},V_{j,i}]$ for every $ij\in E(H)$ and $G[V_{i,0}]$ for every $i\in [m]$ with $w(i)\neq 0$ gives
\begin{align*}
\log\nm(G[V_{i,j},V_{j,i}],\sqrt\varepsilon)&\geq(1-4\sqrt\varepsilon)w(ij)n\log w(ij)n,\\
\log\nm(G[V_{i,0}],\sqrt\varepsilon)&\geq(1-4\sqrt\varepsilon)w(i)n\log w(i)n.
\end{align*}

Therefore,
\[
\nm(G,\sqrt\varepsilon)\geq\prod_{w(ij)\neq0}\nm(G[V_{i,j},V_{j,i}],\sqrt\varepsilon)\prod_{w(i)\neq0}\nm(G[V_{i,0}],\sqrt\varepsilon)\prod_{i=1}^mN(V_i),
\]
which means
\begin{align*}
&\,\log\nm(G,\sqrt\varepsilon)\\
\geq&\,\sum_{w(ij)\neq0}\log\nm(G[V_{i,j},V_{j,i}],\sqrt\varepsilon)+\sum_{w(i)\neq0}\log\nm(G[V_{i,0}],\sqrt\varepsilon)+\sum_{i=1}^m\log N(V_i)\\
\geq&\,(1-4\sqrt\varepsilon)\frac{m}{2}n\log n.
\end{align*}

We now consider the upper bound. Let $\M$ be an arbitrary $\sqrt\varepsilon$-near perfect matching in $G$. For the quotient graph $H$, define $w^\prime(ij)=\frac{|\M\cap E(G[V_i,V_j])|}{n}$ for every $ij\in E(H)$, and $w^\prime(i)=\frac{2|\M\cap E(G[V_i])|}{n}$. It is easy to see that $\Er(H,w^\prime)<\sqrt\varepsilon$. Therefore,
\[
\nm(G,\sqrt\varepsilon)\leq\prod_{w'(ij)\neq0}\nm(G[V_{i,j},V_{j,i}],\sqrt\varepsilon)\prod_{w'(i)\neq0}\nm(G[V_{i,0}],\sqrt\varepsilon)\prod_{i=1}^mN(V_i)\cdot n^{\sqrt\varepsilon n}.
\]
Note that $w'$ may not be a solution of (\ref{eq:1}). Then we have
\begin{align*}
&\,\log\nm(G,\sqrt\varepsilon)\\
\leq&\,\sum_{w'(ij)\neq0}\log\nm(G[V_{i,j},V_{j,i}],\sqrt\varepsilon)+\sum_{w'(i)\neq0}\log\nm(G[V_{i,0}],\sqrt\varepsilon)\\ &\, +\sum_{i=1}^m\log N(V_i)+\sqrt\varepsilon n\log n\\
\leq&\,(1+6\sqrt\varepsilon)\frac{m}{2}n\log n.
\end{align*}
\medskip
	
\CCase{3}{Linear system (\ref{eq:1}) has infinitely many solutions.}

In this case, there exists a positive integer $t$, and variables $x_1,\dots,x_t\in [m]^2$ , such that if we fix the value of $w(x_1),\dots,w(x_t)$, the system of linear equations (\ref{eq:1}) has a unique solution. Let $\bm{x}=(x_1,\dots,x_t)$, and define $\nm(G,\varepsilon,w(\bm{x}))$ to be the number of maximal $\varepsilon$-near perfect matchings $\M$ in $G$, such that for every $i\in [t]$, suppose $x_i$ corresponds $ij$ in $[m]^2$ with $i\neq j$ (or $i=j$), then $|\M\cap G[V_i,V_j]|=w(x_i)n$ (or $|\M\cap G[V_i]|=w(x_i)n/2$).

 Roughly speaking, the number of maximal $\varepsilon$-near perfect matchings in $G$ is about
$$
\int_{[0,1]^t}\nm(G,\varepsilon,w(\bm{x}))\,\mathrm{d}\bm{x}.
$$

In order to avoid double counting, we should be more careful here, since an $\varepsilon$-near perfect matching in $\nm(G,\varepsilon,w(\bm{x}))$ will also be counted in $\nm(G,\varepsilon,w(\bm{x})+\bm{\varepsilon/m^2})$. Let $l=1/\sqrt\varepsilon$, we have
\[
\nm(G,\sqrt\varepsilon)\geq\sum_{i_1,\dots,i_t=0}^l\nm(G,\sqrt\varepsilon,(i_1\sqrt\varepsilon,\dots,i_t\sqrt\varepsilon)).
\]
Then applying the results in Case 2 yields
\[
\log\nm(G,\sqrt\varepsilon)\geq(1-4\sqrt\varepsilon)\frac{m}{2}n\log n.
\]

Considering the upper bound, given an arbitrary $\sqrt\varepsilon$-near perfect matching $\M$, similarly as we did before, define $w'(ij)=\frac{|\M\cap E(G[V_i,V_j])|}{n}$ and $w'(i)=\frac{2|\M\cap E(G[V_i])|}{n}$. This gives $\Er(H,w')<\sqrt\varepsilon$, and therefore, for $n^n\geq(l+1)^{lm}$, we have
\begin{align*}
\log\nm(G,\sqrt\varepsilon)
\leq&\, t\log(l+1)+\log\nm(G,\sqrt\varepsilon,w'(\bm{x}))\\
\leq&\,(1+7\sqrt\varepsilon)\frac{m}{2}n\log n,
\end{align*}
which completes the proof.
\end{proof}

\section{Matchings in dense graphs}
In this section, we analyze the properties of large dense graphs. Suppose $G$ has $n$ vertices. 
After applying Szemer\'edi Regularity Lemma, we have an equitable partition $\mathcal{P}=\{V_1,\dots,V_K\}$. Situation here is more complicated than the one in Section~3, since there can be large matchings between irregular pairs and pairs with low edge densities. We will use the following algorithm by Micali and Vazirani \cite{MV80} to get the size of the maximum matching in graph $G$ and in graph $G[V_i,V_j]$ when $(V_i,V_j)$ is irregular or $\varepsilon$-regular but has low edge density.
\begin{theorem}[\cite{MV80}]\label{thm:maxmatching}
Given a graph $G$, there is a polynomial time algorithm which outputs the size of the maximum matching in $G$, and the running time is $O(\sqrt{|V|}|E|)$.
\end{theorem}


Suppose $H$ is the quotient graph $G/\mathcal{P}$. Let $E_1\subseteq E(H)$ be the set of edges corresponding to the irregular pairs in $G$, $E_2\subseteq E(H)$ be the set of edges corresponding to the $\varepsilon$-regular pairs with edge density at most $n^{\sqrt\varepsilon-1}$, $E_3\subseteq E(H)$ be the set of edges corresponding to the $\varepsilon$-regular pairs with edge density in $[n^{\sqrt\varepsilon-1},n^{-\sqrt\varepsilon}]$ in $G$, and let $E_4\subseteq E(H)$ be the set of edges corresponding to the $\varepsilon$-regular pairs with edge density at least $n^{-\sqrt\varepsilon}$. For every $ij\in E_1$, let $m_{ij}$ be size of the maximum matching in $G[V_i,V_j]$, and let $r_{ij}=Km_{ij}/n$. For every $i\in V(H)$, let $m_i$ be the size of maximum matching in $G[V_i]$, and let $r_i=2Km_i/n$.

Let $Q$ be the graph obtained from $G$ by removing edges inside each $V_i$ and edges between irregular pairs. Suppose $\mathfrak{M}(Q)$ is the set of maximal matchings in $Q$ which can be extended to $\sqrt\varepsilon$-near perfect matchings in $G$. We write $\MM(Q)=|\mathfrak{M}(Q)|$. The following inequalities gives us a way to find $\mathfrak{M}(Q)$ and maximize $\MM(Q)$.

\begin{equation}\label{eq:2}
\begin{split}
&\,0\leq w(e)\leq 1, \quad \text{ for every }e\in E_2\cup E_3\cup E_4,\\
&\,0\leq w(e)\leq r_{e}, \quad\text{ for every }e\in E_1,\\
&\sum_{j\neq i}w(ij)\geq1-r_i-\sqrt\varepsilon,\quad\text{ for every }i\in V(H).
\end{split}
\end{equation}

It is easy to see that if $G$ has $\sqrt\varepsilon$-near perfect matchings, inequality (\ref{eq:2}) has solutions. Define $\mathfrak{S}$ to be the set of feasible solutions of (\ref{eq:2}), and let
\begin{equation}\label{eq:3}
s:=\sup_{w(\mathbf{e})\in\mathfrak{S}}\sum_{e\in E_4}\frac{w(e)n}{K}\log\frac{w(e)n}{K}+\sum_{e\in E_3}\frac{w(e)n}{K}\log p_e\frac{w(e)n}{K}.
\end{equation}

By Theorems \ref{thm:main} and \ref{thm:1q}, we have $\MM(Q)\geq(1-4\sqrt\varepsilon)s$, which means $\nm(G,\sqrt\varepsilon)\geq(1-4\sqrt\varepsilon)s$, and this proves Theorem \ref{thm:1.7}. 

With all tools in hand, we are going to state the algorithm {\sc Number of Max Near Perfect Matchings Dense}. Given a graph $G$ of order $n$ and a real number $\varepsilon>0$, we do the following:

\begin{algorithm}[H]
\SetAlgoLined
\SetKwInOut{Input}{Input}
 \medskip
 
 {\sc Step 1.} Apply the algorithm in Theorem \ref{thm:maxmatching} on $G$. If $G$ does not contain any $\varepsilon$-near perfect matchings, output $0$. Otherwise, do the following steps.
 \smallskip
 
 {\sc Step 2.} Take $\tau=3\varepsilon^2/2$, and $\alpha=1/2$, $h=1/\varepsilon$. Apply the algorithm in Theorem~\ref{thm:algpartition} with integer $k$ taking values from $h$ to $M(h,\tau)$. Then the algorithm will output an $\varepsilon^2$-Szemer\'edi partition into $K$ parts, with $h\leq K\leq M(h,\tau)$.
 \smallskip

{\sc Step 3.} Apply the algorithm in Theorem \ref{thm:maxmatching} at most $K^2$ times, to compute the size of maximum matchings to obtain $r_i$ and $r_{ij}$. Solve the inequalities (\ref{eq:2}) and compute the value of $s$ in (\ref{eq:3}). Let $\ell=(1-4\varepsilon)s$, then output $n^{\ell n}$.
\medskip
  
 \caption{{\sc Number of Max Near Perfect Matchings Dense}} 
\end{algorithm}

\medskip

The above algorithm provides a lower bound for the number of maximal near perfect matchings, and its running time is $O(n^{5/2})$. Unfortunately, the lower bound we obtain is not tight. Let us illustrate on this by an example.

Suppose $G$ is a dense graph of order $n$ together with a Szemer\'edi partition $\mathcal{P}=V_1,\dots,V_K$, each of size $n/K$, where $K\geq2/\varepsilon$ and suppose $K\equiv2\mod4$. Induced graphs between all the pairs $(V_i,V_j)$ are $\varepsilon$-regular except $K/2\leq\varepsilon K^2$ irregular pairs $(V_i,V_{i+1})$ for $i=1,3,5,\dots,K/2$. Graphs $G[V_i]$ are empty for $1\leq i\leq (K+2)/2$, and graphs $G[V_i,V_j]$ are empty when $i\leq(K+2)/2$ and $j\neq i+1$ when $i$ is odd, $j\neq i-1$ when $i$ is even. All the vertices in $V_i$ form a large complete graph for $i\geq (K+4)/2$, and graphs $G[V_i,V_{i+1}]$ are complete bipartite for $i=1,3,5,\dots,K/2$.

Now, it is easy to see that the number of perfect matchings in $G$ is $n^{n/2}$. After we remove edges between irregular pairs, we remove $\frac{K}{2}(\frac{n}{K})^2<\varepsilon n^2$ edges. Then the number of extend-able maximal matchings in the obtained graph (the output of the above algorithm) is $n^{n/4}$, we lose a factor $n^{n/4}$.

\section*{Acknowledgements}
The authors would like to thank Andrei Bulatov, Bojan Mohar and Fan Wei for many helpful discussions. We are also thankful to Heng Guo for pointing out an inaccuracy in the introduction, and bringing reference~\cite{PR17} to our attention after the first version of this paper appeared on arXiv.

\appendix
\section{}

\begin{lemma}\label{lem:a1}
Given $\varepsilon>0$, $\delta=\sqrt\varepsilon/p$ and $c=1/(1-\delta)p$. Suppose
\[
k=\frac{1}{\sqrt\varepsilon}-c, \qquad t=\frac{\log(1-\delta)p}{\log(1-(1-\delta)p)},
\]
and \[
\nm(G,\sqrt\varepsilon)\geq\prod_{i=0}^{k-1}\frac{\big((1-\delta)p(n-i\sqrt\varepsilon n)\big)!}{\big((1-\delta)p(n-i\sqrt\varepsilon n)-\sqrt\varepsilon n\big)!}\prod_{i=0}^{t-1}\big((1-\delta)pc\sqrt\varepsilon n(1-(1-\delta)p)^i\big)!.
\]
Then we have $\log \nm(G,\sqrt\varepsilon)\geq (1-3\sqrt\varepsilon)n\log pn.$
\end{lemma}
\begin{proof}
We have
\begin{align*}
&\,\log\nm(G,\sqrt\varepsilon)\\
\geq&\,\sum_{i=0}^{k-1}(1-\delta)p(n-i\sqrt\varepsilon n)\log \big((1-\delta)p(n-i\sqrt\varepsilon n)\big)\\
&\, -\sum_{i=0}^{k-1}\big((1-\delta)p(n-i\sqrt\varepsilon n)-\sqrt\varepsilon n\big)\log\big((1-\delta)p(n-i\sqrt\varepsilon n)-\sqrt\varepsilon n\big)\\
&\, +\sum_{i=0}^{t-1}(1-\delta)pc\sqrt\varepsilon n(1-(1-\delta)p)^i\log\big((1-\delta)pc\sqrt\varepsilon n(1-(1-\delta)p)^i\big)\\
=&\,(1-\delta)p\frac{nk(1+c\sqrt\varepsilon)}{2}\log (1-\delta)pn+\sum_{i=0}^{k-1}(1-\delta)pn(1-i\sqrt\varepsilon)\log (1-i\sqrt\varepsilon)\\
&\,-\big((1-\delta)p\frac{nk(1+c\sqrt\varepsilon)}{2}-k\sqrt\varepsilon n\big)\log(1-\delta)pn\\
&\,-\sum_{i=0}^{k-1}\big((1-\delta)pn(1-i\sqrt\varepsilon)-\sqrt\varepsilon n\big)\log\Big((1-i\sqrt\varepsilon)-\frac{\sqrt\varepsilon}{(1-\delta)p}\Big)\\
&\,+\frac{1-(1-(1-\delta)p)^{t+1}}{(1-\delta)p}(1-\delta)pc\sqrt\varepsilon n\log(1-\delta)pn\\
&\,+\sum_{i=0}^{t-1}(1-\delta)pc\sqrt\varepsilon n(1-(1-\delta)p)^i\log\big(c\sqrt\varepsilon (1-(1-\delta)p)^i\big)\\
\geq&\, k\sqrt\varepsilon n\log (1-\delta)pn+(1-(1-\delta)p)c\sqrt\varepsilon n\log (1-\delta)pn)\\
&\, +\sum_{i=0}^{k-1}\sqrt\varepsilon n\log\big((1-i\sqrt\varepsilon)-c\sqrt\varepsilon\big)+\frac{1}{2}\sum_{i=0}^{t-1}\sqrt\varepsilon n(1-(1-\delta)p)^i\log\varepsilon\\
\geq&\,(1-2\sqrt\varepsilon)n\log(1-\delta)pn+k\sqrt\varepsilon n\log\sqrt\varepsilon+(c-1)\sqrt\varepsilon n\log\sqrt\varepsilon\\
=&\,(1-2\sqrt\varepsilon)n\log(1-\delta)\sqrt\varepsilon pn
>(1-3\sqrt\varepsilon)n\log pn.
\end{align*}
\end{proof}

\bibliographystyle{abbrv}
\bibliography{reference}

\begin{thebibliography}{10}

\bibitem{N94}
N.~Alon, R.~A. Duke, H.~Lefmann, V.~R\"odl, and R.~Yuster.
\newblock The algorithmic aspects of the regularity lemma.
\newblock {\em J. Algorithms}, 16(1):80--109, 1994.

\bibitem{arnborg1991easy}
S.~Arnborg, J.~Lagergren, and D.~Seese.
\newblock Easy problems for tree-decomposable graphs.
\newblock {\em J. Algorithms}, 12(2):308--340, 1991.

\bibitem{bapat2007recent}
R.~Bapat.
\newblock Recent developments and open problems in the theory of permanents.
\newblock {\em The Mathematics student}, 76(1):55, 2007.

\bibitem{bayati2007simple}
M.~Bayati, D.~Gamarnik, D.~Katz, C.~Nair, and P.~Tetali.
\newblock Simple deterministic approximation algorithms for counting matchings.
\newblock In {\em Proceedings of the thirty-ninth annual ACM symposium on
  Theory of computing}, pages 122--127. ACM, 2007.

\bibitem{broder1986hard}
A.~Z. Broder.
\newblock How hard is it to marry at random?(on the approximation of the
  permanent).
\newblock In {\em Proceedings of the eighteenth annual ACM symposium on Theory
  of computing}, pages 50--58. ACM, 1986.

\bibitem{chien2004determinant}
S.~Chien.
\newblock A determinant-based algorithm for counting perfect matchings in a
  general graph.
\newblock In {\em Proceedings of the fifteenth annual ACM-SIAM symposium on
  Discrete algorithms}, pages 728--735. Society for Industrial and Applied
  Mathematics, 2004.

\bibitem{quasi}
F.~R.~K. Chung, R.~L. Graham, and R.~M. Wilson.
\newblock Quasi-random graphs.
\newblock {\em Combinatorica}, 9(4):345--362, 1989.

\bibitem{csikvari2014lower}
P.~Csikv{\'a}ri.
\newblock Lower matching conjecture, and a new proof of schrijver's and
  gurvits's theorems.
\newblock {\em arXiv preprint arXiv:1406.0766}, 2014.

\bibitem{dagum1992approximating}
P.~Dagum and M.~Luby.
\newblock Approximating the permanent of graphs with large factors.
\newblock {\em Theoretical Computer Science}, 102(2):283--305, 1992.

\bibitem{diaconis2001statistical}
P.~Diaconis, R.~Graham, and S.~P. Holmes.
\newblock Statistical problems involving permutations with restricted
  positions.
\newblock {\em Lecture Notes-Monograph Series}, pages 195--222, 2001.

\bibitem{dyer2017switch}
M.~Dyer, M.~Jerrum, and H.~M{\"u}ller.
\newblock On the switch markov chain for perfect matchings.
\newblock {\em Journal of the ACM (JACM)}, 64(2):12, 2017.

\bibitem{dyer2017counting}
M.~Dyer and H.~M{\"u}ller.
\newblock Counting perfect matchings and the switch chain.
\newblock {\em arXiv preprint arXiv:1705.05790}, 2017.

\bibitem{fisher1961statistical}
M.~E. Fisher.
\newblock Statistical mechanics of dimers on a plane lattice.
\newblock {\em Physical Review}, 124(6):1664, 1961.

\bibitem{FLZ}
J.~Fox, L.~M. Lov\'asz, and Y.~Zhao.
\newblock On regularity lemmas and their algorithmic applications.
\newblock {\em Combin. Probab. Comput.}, 26(4):481--505, 2017.

\bibitem{friedland2008number}
S.~Friedland, E.~Krop, and K.~Markstr{\"o}m.
\newblock On the number of matchings in regular graphs.
\newblock {\em the electronic journal of combinatorics}, 15(1):110, 2008.

\bibitem{furer2005approximately}
M.~F{\"u}rer and S.~P. Kasiviswanathan.
\newblock Approximately counting perfect matchings in general graphs.
\newblock In {\em ALENEX/ANALCO}, pages 263--272. Citeseer, 2005.

\bibitem{galluccio1999theory}
A.~Galluccio and M.~Loebl.
\newblock On the theory of pfaffian orientations. i. perfect matchings and
  permanents.
\newblock {\em Electron. J. combin}, 6(1):R6, 1999.

\bibitem{jerrum2003counting}
M.~Jerrum.
\newblock {\em Counting, sampling and integrating: algorithms and complexity}.
\newblock Springer Science \& Business Media, 2003.

\bibitem{jerrum1989approximating}
M.~Jerrum and A.~Sinclair.
\newblock Approximating the permanent.
\newblock {\em SIAM journal on computing}, 18(6):1149--1178, 1989.

\bibitem{jerrum1996markov}
M.~Jerrum and A.~Sinclair.
\newblock The markov chain monte carlo method: an approach to approximate
  counting and integration.
\newblock {\em Approximation algorithms for NP-hard problems}, pages 482--520,
  1996.

\bibitem{jerrum2004polynomial}
M.~Jerrum, A.~Sinclair, and E.~Vigoda.
\newblock A polynomial-time approximation algorithm for the permanent of a
  matrix with nonnegative entries.
\newblock {\em Journal of the ACM (JACM)}, 51(4):671--697, 2004.

\bibitem{kasteleyn1961statistics}
P.~W. Kasteleyn.
\newblock The statistics of dimers on a lattice: I. the number of dimer
  arrangements on a quadratic lattice.
\newblock {\em Physica}, 27(12):1209--1225, 1961.

\bibitem{lelarge2017counting}
M.~Lelarge.
\newblock Counting matchings in irregular bipartite graphs and random lifts.
\newblock In {\em Proceedings of the Twenty-Eighth Annual ACM-SIAM Symposium on
  Discrete Algorithms}, pages 2230--2237. Society for Industrial and Applied
  Mathematics, 2017.

\bibitem{MV80}
S.~Micali and V.~V. Vazirani.
\newblock An {$O(\sqrt{|V |}|E|)$} algorithm for finding maximum matching in
  general graphs.
\newblock {\em 1980 {IEEE} 21th {A}nnual {S}ymposium on {F}oundations of
  {C}omputer {S}cience---{FOCS} 1980, \emph{{IEEE} {C}omputer {S}oc.,
  {S}yracuse, {NY}}}, 1980.

\bibitem{okamoto2009counting}
Y.~Okamoto, R.~Uehara, and T.~Uno.
\newblock Counting the number of matchings in chordal and chordal bipartite
  graph classes.
\newblock In {\em International Workshop on Graph-Theoretic Concepts in
  Computer Science}, pages 296--307. Springer, 2009.

\bibitem{PR17}
V.~Patel and G.~Regts.
\newblock Deterministic polynomial-time approximation algorithms for partition
  functions and graph polynomials.
\newblock {\em SIAM Journal on Computing}, 46(6):1893--1919, 2017.

\bibitem{ryser1963combinatorial}
H.~J. Ryser.
\newblock Combinatorial mathematics, carus mathematical monographs, no. 14.
\newblock {\em Math. Assoc. America}, 1963.

\bibitem{schrijver1998counting}
A.~Schrijver et~al.
\newblock Counting 1-factors in regular bipartite graphs.
\newblock {\em J. Comb. Theory, Ser. B}, 72(1):122--135, 1998.

\bibitem{sinclair2012algorithms}
A.~Sinclair.
\newblock {\em Algorithms for random generation and counting: a Markov chain
  approach}.
\newblock Springer Science \& Business Media, 2012.

\bibitem{S}
E.~Szemer\'edi.
\newblock On sets of integers containing no {$k$} elements in arithmetic
  progression.
\newblock {\em Acta Arith.}, 27:199--245, 1975.
\newblock Collection of articles in memory of Juri\u\i Vladimirovi\v c Linnik.

\bibitem{szemeredi1975regular}
E.~Szemer{\'e}di.
\newblock Regular partitions of graphs.
\newblock Technical report, Stanford {U}niv {C}alif {D}ept of {C}omputer
  {S}cience, 1975.

\bibitem{Tao10}
T.~Tao.
\newblock {\em An epsilon of room, {II}}.
\newblock American Mathematical Society, Providence, RI, 2010.
\newblock Pages from year three of a mathematical blog.

\bibitem{temperley1961dimer}
H.~N. Temperley and M.~E. Fisher.
\newblock Dimer problem in statistical mechanics-an exact result.
\newblock {\em Philosophical Magazine}, 6(68):1061--1063, 1961.

\bibitem{tesler2000matchings}
G.~Tesler.
\newblock Matchings in graphs on non-orientable surfaces.
\newblock {\em Journal of Combinatorial Theory, Series B}, 78(2):198--231,
  2000.

\bibitem{quasi2}
A.~Thomason.
\newblock Pseudorandom graphs.
\newblock In {\em Random graphs '85 ({P}ozna\'n, 1985)}, volume 144 of {\em
  North-Holland Math. Stud.}, pages 307--331. North-Holland, Amsterdam, 1987.

\bibitem{vadhan2001complexity}
S.~P. Vadhan.
\newblock The complexity of counting in sparse, regular, and planar graphs.
\newblock {\em SIAM Journal on Computing}, 31(2):398--427, 2001.

\bibitem{valiant1979complexity}
L.~G. Valiant.
\newblock The complexity of computing the permanent.
\newblock {\em Theoretical computer science}, 8(2):189--201, 1979.

\bibitem{zdeborova2006number}
L.~Zdeborov{\'a} and M.~M{\'e}zard.
\newblock The number of matchings in random graphs.
\newblock {\em Journal of Statistical Mechanics: Theory and Experiment},
  2006(05):P05003, 2006.

\end{thebibliography}
\end{document}